\documentclass[a4paper,10pt]{amsart}

\usepackage{microtype}
\usepackage{amsthm,amsmath,amssymb,graphicx}
\usepackage{graphicx,subcaption}
\usepackage{tikz}

\usetikzlibrary{arrows,backgrounds,decorations,automata}
\newtheorem{theorem}{Theorem}
\newtheorem{lemma}{Lemma}
\newtheorem{proposition}{Proposition}

\newtheorem{corollary}{Corollary}

\theoremstyle{definition}
\newtheorem{definition}{Definition}

\newtheorem{remark}{Remark}

\makeatletter
\DeclareRobustCommand\bigop[2][1]{%
  \mathop{\vphantom{\sum}\mathpalette\bigop@{{#1}{#2}}}\slimits@
}
\newcommand{\bigop@}[2]{\bigop@@#1#2}
\newcommand{\bigop@@}[3]{%
  \vcenter{%
    \sbox\z@{$#1\sum$}%
    \hbox{\resizebox{\ifx#1\displaystyle#2\fi\dimexpr\ht\z@+\dp\z@}{!}{$\m@th#3$}}%
  }%
}
\makeatother

\newcommand{\bigast}{\DOTSB\bigop{\ast}}
\newcommand{\ee}{\mathrm{e}}
\newcommand{\ii}{\mathrm{i}}
\newcommand{\bA}{{\bf A}}
\newcommand{\bS}{{\bf S}}

\newcommand{\exend}{\hfill $\Diamond$}

\newcounter{nootje}
\newcommand\noot[1]
 {\marginpar{\footnotesize\begin{minipage}{30mm}\begin{flushleft}\thenootje : #1\end{flushleft}\end{minipage}}\addtocounter{nootje}{1}}
\begin{document}

\title[The spectral theory of regular sequences]{The spectral theory of regular sequences}


\author{Michael Coons}
\author{James Evans}
\address{School of Information and Physical Sciences, 
   University of Newcastle, \newline
\hspace*{\parindent}University Drive, Callaghan NSW 2308, Australia}
\email{michael.coons@newcastle.edu.au, james.evans10@uon.edu.au}

\author{Neil Ma\~nibo} 
\address{Fakult\"at f\"ur Mathematik, Universit\"at Bielefeld, \newline
\hspace*{\parindent}Postfach 100131, 33501 Bielefeld, Germany}
\email{cmanibo@math.uni-bielefeld.de }

\date{\today}

\begin{abstract} Regular sequences are natural generalisations of fixed points of constant-length substitutions on finite alphabets, that is, of automatic sequences. Using the harmonic analysis of measures associated with substitutions as motivation, we study the limiting asymptotics of regular sequences by constructing a systematic measure-theoretic framework surrounding them. The constructed measures are generalisations of mass distributions supported on attractors of iterated function systems.\end{abstract}

\maketitle

\section{Introduction}

A sequence $f$ is called \emph{$k$-automatic} if there is a deterministic finite automaton that takes in the base-$k$ expansion of a positive integer $n$ and outputs the value $f(n)$. Automatic sequences are ubiquitous in number theory, theoretical computer science and symbolic dynamics, and can be described in many ways, though the one we find most convenient is via the $k$-kernel, $${\rm ker}_k ({f}):=\left\{(f(k^\ell n+r))_{n\geqslant 0}: \ell\geqslant 0, 0\leqslant r<k^\ell\right\}.$$ A sequence $f$ is $k$-automatic if and only if its $k$-kernel is finite \cite[Prop.~V.3.3]{E1974}. It is immediate that an automatic sequence takes only a finite number of values. A natural generalisation to sequences that can be unbounded was given in the early nineties by Allouche and Shallit \cite{AS1992}; an integer sequence $f$ is called \emph{$k$-regular} if the $\mathbb{Q}$-vector space $$\mathcal{V}_k(f):=\langle{\rm ker}_k ({f})\rangle_\mathbb{Q}$$ generated by the $k$-kernel of $f$ is finite-dimensional over $\mathbb{Q}$. One nice property of this generalisation is that a bounded regular sequence is automatic. Additionally, the set of $k$-regular sequences has algebraic structure, it forms a ring under point-wise addition and (Cauchy) convolution. 

The study of automatic sequences is rich from both number-theoretical and dynamical viewpoints. Much of the number-theoretic literature on automatic sequences mirrors that of the rational-transcendental dichotomies of integer power series proved in the first third of the twentieth century, such as those of Fatou \cite{F1906} and Szeg\H{o} \cite{S1922}, and culminating with the work of Carlson \cite{C1921}, who proved that if $F(z)\in\mathbb{Z}[[z]]$ converges in the unit disc, then either $F(z)$ has the unit circle as a natural boundary or it is the power series expansion of a rational function. The most celebrated result for automatic sequences is the resolution of the Cobham--Loxton--van der Poorten Conjecture by Adamczewski and Bugeaud \cite{AB2007}---a real number whose base expansion is given by an automatic sequence is either rational or transcendental. 

The dynamical literature has focussed on the study of automatic sequences through their related substitution systems---every automatic sequence is a coding of an infinite fixed point of a constant-length substitution on finitely many letters; see Cobham \cite{C1972}. The long-range order of substitution systems has been well studied; there is an abundance of literature on this still very active area of research going back to the seminal works of Wiener \cite{W1927} and Mahler \cite{M1927}. The monographs by Queff\'elec \cite{Qbook} and Baake and Grimm \cite{BGbook} contain details about both the tiling and symbolic pictures of these systems as well as the associated diffraction and spectral measures---the modern-classical means of examining the long-range order of these systems. The spectral results concerning substitution systems are not dichotomies, but classifications based on the Lebesgue Decomposition Theorem. In this context, one starts with a substitution and forms a measure---usually the spectral measure or the diffraction measure. The culminating result is then determining the spectral type of the measure. The Thue--Morse dynamical system is a paradigmatic example; the Thue--Morse sequence is $2$-automatic and the fixed point of the substitution on two letters given by 
\begin{equation*}
   \varrho_{_{\rm TM}}: \, \begin{cases} a \mapsto ab \\ 
     b\mapsto ba\, . \end{cases}
\end{equation*} 
Using appropriate weights, the diffraction measure of the Thue--Morse sequence is purely singular continuous with respect to Lebesgue measure. This result---the first to explicitly record a singular continuous measure---was proved by Mahler \cite{M1927}, and later in a dynamical setting by Kakutani~\cite{K1972}. For details on diffraction see Baake and Grimm \cite[Ch.~9]{BGbook}.

Transitioning to regular sequences, the number-theoretic story is much the same as automatic sequences, mirroring that of rational-transcendental dichotomies. The generalisation of the Cobham--Loxton--van der Poorten Conjecture for regular sequences was proved by Bell, Bugeaud and Coons \cite{BBC2015}. But, in contrast to automatic sequences, the study of the long-range order of unbounded regular sequences $f$, and so also the related spaces $\mathcal{V}_k(f)$, is not so straight-forward; neither diffraction nor spectral measures can be associated with them in a natural way. 

As a first step of addressing the long-range order of such objects, Baake and Coons~\cite{BC2018} introduced a natural probability measure associated with Stern's diatomic sequence, one simple example of an unbounded regular sequence. Stern's diatomic sequence ${s}$ is given by $s(0)=0$, $s(1)=1$, and for $n\geqslant 1$ by the recurrences $s(2n)=s(n)$ and $s(2n+1)=s(n)+s(n+1)$. Its $2$-kernel ${\rm ker}_2 ({s})$ is infinite, but induces a finitely generated vector space; $$\mathcal{V}_2(s)=\langle{\rm ker}_2 ({s})\rangle_{\mathbb{Q}}=\langle\{(s(n))_{n\geqslant 0},(s(n+1))_{n\geqslant 0}\}\rangle_{\mathbb{Q}}.$$ Stern's diatomic sequence has many interesting properties, e.g., the sequence of ratios $(s(n)/s(n+1))_{n\geqslant 0}$ enumerates the non-negative rationals in reduced form and without repeats, and---like the diffraction measure of the Thue--Morse sequence---the associated measure is singular continuous. Baake's and Coons's result rests on the fact that the sequence ${s}$ satisfies certain self-similar type properties; it has a fundamental region of recursion---between consecutive powers of two---and the sum of ${s}$ over this fundamental region is linearly recurrent, enabling the use of a volume-averaging process. 

In this paper, we provide the complete generalisation of the result of Baake and Coons~\cite{BC2018}. This entails not only taking a $k$-regular sequence $f$ and associating a natural probability measure $\mu_f$ to $\mathcal{V}_k(f)$, but also determining the class of regular sequences $f$ for which this is possible. Indeed, there are regular sequences for which such measures do not exist, and for which the sequence of approximants is not even eventually periodic.

The measures $\mu_f$ we construct are analogous to mass distributions supported on attractors of certain iterated function systems---those constructed by repeated subdivision; see Falconer~\cite{Fbook} for related definitions and Coons and Evans~\cite{CEpre} for a family of extended examples related to generalised Cantor sets. Here, rather than iterating intervals under a finite number of maps, the finite approximants are pure point measures on the unit interval, whose corresponding weights possess the recursive structure. 
In this way, the existence of a natural measure associated with $f$ and $\mathcal{V}_k(f)$ provides a path to viewing the space $\mathcal{V}_k(f)$ of regular sequences associated with $f$ as a dynamical system and opens the possibility of considering these sequences and spaces as (fractal) geometric structures.

To see how this works, we start with an integer-valued $k$-regular sequence $f$ and obtain a basis for $\mathcal{V}_k(f)$, the $\mathbb{Q}$-vector space generated by the $k$-kernel of $f$, which is maximal with respect to asymptotic growth over certain fundamental regions. In particular, let $k\geqslant 2$ be an integer, $f$ be a $k$-regular sequence and let the set of integer sequences \begin{equation}\label{basis}\{f=f_1,f_2,\ldots,f_d\}\subseteq {\rm ker}_k ({f})
\end{equation} be a basis for $\mathcal{V}_k(f):=\langle{\rm ker}_k ({f})\rangle_\mathbb{Q}$. Set ${\bf f}(m)=(f_1(m),f_2(m),\ldots,f_d(m))^T$. For each $a\in\{0,\ldots,k-1\}$ let ${\bf B}_a$ be the $d\times d$ integer matrix such that, for all $m\geqslant 0$, 
\begin{equation}\label{eq:k kernel}
{\bf f}(km+a)={\bf B}_a\, {\bf f}(m).
\end{equation} 
We refer the reader to the seminal paper of Allouche and Shallit \cite{AS1992} and Nishioka's monograph \cite[Ch.~15]{N1996} for details on existence and the finer definitions. Note that there is ${\bf w}\in\mathbb{Z}^{d\times 1}$ such that for each $i\in\{1,\ldots,d\}$ and $n>0$, we have $$f_i(m)=e_i^T {\bf B}_{(m)_k} {\bf w}=e_i^T{\bf B}_{i_0}{\bf B}_{i_1}\cdots {\bf B}_{i_s}{\bf w},$$ where $e_i$ is the $i$th elementary column vector, $(m)_k=i_s\cdots i_1i_0$ is the base-$k$ expansion of $m$ and ${\bf B}_{(m)_k}:={\bf B}_{i_0}{\bf B}_{i_1}\cdots {\bf B}_{i_s}.$ Set $${\bf B}:=\sum_{a=0}^{k-1}{\bf B}_a.$$ 

As we start to view the vector space $\mathcal{V}_k(f)$ dynamically, we note that the matrix ${\bf B}$ is analogous to the substitution matrix ${\bf M}_\varrho$ of a substitution $\varrho$ on a finite alphabet. In fact, if $\varrho$ is a constant-length substitution of length $k$ on the $d$ letters $0,1,\ldots,d-1$, then the matrices ${\bf B}_a$ are the so-called digit (instruction) matrices and ${\bf M}_\varrho={\bf B}$; compare \cite{BGM2019,Qbook}. When considering the dynamical properties of substitution systems it is common to assume that the substitution is primitive; that is, the non-negative substitution matrix ${\bf M}_\varrho$ is primitive. To continue our analogy with substitutions, we make similar assumptions. At first glance, it is reasonable to restrict ourselves to the assumption that ${\bf B}$ is primitive. But if this is the case for the $k$-regular sequence $f$, then we can consider $f$ as a $k^j$-regular sequence with $j$ being the smallest positive integer for which ${\bf B}^j$ is positive. Thus, in the context of $k$-regular sequences the distinction between primitivity and positivity is somewhat blurred. Note that this is also tacitly done for substitutions, where one normally chooses an appropriate power $j$ such that ${\bf M}_{\varrho}^{j}>0$, or equivalently, $\varrho^{j}(a)$ contains all the letters of the alphabet $\mathcal{A}$, for all $a\in \mathcal{A}$. Hence we make the following definition.

\begin{definition}\label{def:prim} We call a $k$-regular sequence $f$ {\em primitive} provided $f$ takes non-negative integer values, is not eventually zero, each of the $k$ digit matrices ${\bf B}_a$ are non-negative and the matrix ${\bf B}$ is positive.
\end{definition}

\noindent In Section \ref{nonprim}, we discuss the necessity of these assumptions. In particular, we present a sequence whose associated matrix is not primitive and show that our construction does not converge to a measure for this sequence.

We are close to being able to state our results---the final ingredients are the introduction of fundamental regions and the definition of the related pure point measures. To this end, we note that, in analogy to the lengths of the iterates of a substitution applied to an initial seed, a $k$-regular sequence $f$ has a fundamental region, the interval $[k^n,k^{n+1})$. Over these regions, the sum of a regular sequence is linearly recurrent, a property which can be thought of as a sort of self-similarity property for regular sequences. This property provides a structured volume to average by, and using it, we can construct a probability measure. Formally, for each $i\in\{1,\ldots,d\}$, set \begin{equation}\label{eq:Sigmai}\varSigma_i(n):=\sum_{m=k^n}^{k^{n+1}-1}f_i(m)\end{equation} and define
\begin{equation}\label{eq:def-meas}
   \mu^{}_{n,i} \, := \, \frac{1}{\varSigma_i(n)} \sum_{m=0}^{k^{n+1}-k^n - 1}
   f_i(k^n + m) \, \delta^{}_{m / k^n(k-1)},
\end{equation}
where $\delta_x$ denotes the unit Dirac measure at $x$. We can view
$(\mu^{}_{n,i})^{}_{n\in\mathbb{N}_0}$ as a sequence of probability measures on
the $1$-torus, the latter written as $\mathbb{T}=[0,1)$ with addition modulo
$1$. Here, we have simply re-interpreted the (normalised) values of 
the sequence $(f_i(m))_{m\geqslant 0}$ between $k^n$ and $k^{n+1}-1$ as the weights of a pure point
probability measure on $\mathbb{T}$ supported on the set $\big\{ \frac{m}{k^n(k-1)} : 0 \leqslant m <
k^n(k-1) \big\}$. Set \begin{equation}\label{eq:mun}\boldsymbol{\mu}_{n}=\boldsymbol{\mu}_{f,n}:=(\mu_{n,1},\ldots,\mu_{n,d})^T.\end{equation}

Our first result, Theorem \ref{main1} below, provides a unique probability measure on $\mathbb{T}$ associated with the space $\mathcal{V}_k(f)$. 

\begin{theorem}\label{main1} Let $f$ be a primitive $k$-regular sequence. If $f_1,\ldots,f_d$ form the basis of $\mathcal{V}_k(f)$ associated with ${\bf B}$, then the vectors $\boldsymbol{\mu}_{f,n}$ of pure point measures converges weakly to a vector of probability measures $\boldsymbol{\mu}_f=(\mu_f,\ldots,\mu_f)^T$ on $\mathbb{T}$. 
\end{theorem}

Of course, one may also wish to consider the measure associated with $f$ apart from the full considerations surrounding that of the space $\mathcal{V}_k(f)$. With a little more specificity one can prove a stronger result, which is roughly as follows.

\begin{theorem}\label{main2} Let $f$ be a positive integer-valued $k$-regular sequence such that the spectral radius, $\rho({\bf B})$, is the unique dominant eigenvalue of ${\bf B}$ and the joint spectral radius, $\rho^*(\{{\bf B}_0,\ldots,{\bf B}_{k-1}\})<\rho({\bf B})$. Then the measure $\mu_f$ exists and is continuous. 
\end{theorem}

While, in Theorems \ref{main1} and \ref{main2} we have assumed that $f$ is non-negative, in Theorem~\ref{main2} this assumption can be traded for a more general condition. In Section \ref{sec:f}, we prove a slightly stronger result, where we replace this non-negativity assumption with two conditions both having to do with a related dilation equation. The first condition is  a natural non-vanishing condition and the second is that the solution of the associated dilation equation is of bounded variation. This more general situation may not result in a probability measure, but possibly, a signed measure.

This article is organised as follows. Section~\ref{sec:V} contains the proof of our first theorem, the existence of a measure associated with $\mathcal{V}_k(f)$, and Section~\ref{sec:f} contains the general technical result that implies our second theorem. In Section~\ref{nonprim}, we discuss the necessity of our assumptions and provide witnessing examples. Finally, we offer some concluding remarks and open questions in Section~\ref{sec:cr}.

\section{A natural probability measure associated with $\mathcal{V}_k(f)$}\label{sec:V}

The aim of this section is to prove Theorem \ref{main1}. To do this, we attempt to mimic the ideas behind the establishment of spectral measures associated with substitution dynamical systems. Fortunately, the generating power series of $k$-regular sequences satisfy functional equations that can be thought of as taking the place of a substitution---the generating power series of the sequences $f_i$ in \eqref{basis} are Mahler functions. That is, for each $i\in\{1,\ldots,d\}$ setting $F_i(z):=\sum_{{m}\geqslant 0}f_i({m}) z^{{m}}$ and ${\bf F}(z)=(F_1(z),\ldots,F_d(z))^T$, we have \cite[p.~153]{N1996} that ${\bf F}(z)$ satisfies the Mahler functional equation \begin{equation}\label{MFE}{\bf F}(z)={\bf B}(z)\, {\bf F}(z^k),\end{equation} where ${\bf B}(z):=\sum_{a=0}^{k-1}{\bf B}_a\, z^a.$ The matrix-valued function ${\bf B}(z)$ is analogous to the Fourier cocycle in the renormalisation theory of substitution and inflation systems, which carries information about features of the underlying diffraction and spectral measures; see Bufetov and Solomyak \cite{BS2014} and Baake, G\"ahler and Ma\~nibo~\cite{BGM2019}. Note that ${\bf B}(1)={\bf B}$; this specialisation will be discussed more below. Equation \eqref{MFE} shows that the functions $F_i(z)$ behave well under the Frobenius map $z\mapsto z^k$. The functional equation \eqref{MFE} is analogous to a substitution with repeated application mirroring the iterated composition of a substitution. This property is essentially what allows us to form certain cocycles (e.g., see \eqref{AmSm}) that, under the primitivity assumption above, have convergence properties which provide for the existence of the desired limit measures. To achieve our goal, we require a few preliminary results.

For $i\in\{0,\ldots, k-1\}$ let $\varSigma_i(n)$ be as defined in \eqref{eq:Sigmai} and set $${\bf\Sigma}(n):=(\varSigma_1(n),\varSigma_2(n),\ldots,\varSigma_d(n))^T.$$ In the Introduction, we stated that the sequences $\varSigma_i(n)$ are linearly recurrent; we prove this result as the following lemma.

\begin{lemma}\label{lem:Sigma n}
If $n\geqslant 1$, then ${\bf\Sigma}(n)={\bf B}\cdot {\bf\Sigma}(n-1).$
\end{lemma}

\begin{proof} Here we are considering the sums of each $f_i$ over all integers in the interval $[k^n,k^{n+1}-1]$. These are precisely all of the integers that have $n+1$ digits in their $k$-ary expansions. Noting that the $k$-ary expansion of a non-zero integer cannot begin with a zero, we thus have that $\varSigma_i(n)$ satisfies  $$\sum_{m=k^n}^{k^{n+1}-1}f_i(m) =\sum_{m=k^n}^{k^{n+1}-1}e_i^T {\bf B}_{(m)_k}{\bf w}=e_i^T \left({\bf B}_0+{\bf B}_1+\cdots+{\bf B}_{k-1}\right)^n\sum_{a=1}^{k-1}{\bf B}_{a} {\bf w}.$$
Thus \begin{align*}{\bf\Sigma}(n)&=\left({\bf B}_0+{\bf B}_1+\cdots+{\bf B}_{k-1}\right)^n\sum_{a=1}^{k-1}{\bf B}_{a} {\bf w}\\ &=\left({\bf B}_0+{\bf B}_1+\cdots+{\bf B}_{k-1}\right){\bf\Sigma}(n-1),\end{align*} where the second equality follows from the first by writing $n=1+(n-1)$. 
\end{proof}

Now, define the polynomials $b_{ij}(z)$ by ${\bf B}(z)=\left(b_{ij}(z)\right)_{1\leqslant i,j\leqslant d}$ and for each $n\geqslant 1$ define the matrix 
\begin{equation}\label{eq:A matrix}
{\bf A}_n(z):=\left(\frac{\varSigma_j(n-1)\, b_{ij}(z)}{\varSigma_i(n)}\right)_{1\leqslant i,j\leqslant d}.
\end{equation} Note that since $f$ is primitive, the denominator $\varSigma_i(n)$ cannot vanish. The matrices ${\bf A}_n(z)$ are normalised versions of the matrix ${\bf B}(z)$. In particular, they allow us to lift the result of Lemma \ref{lem:Sigma n} to the level of measures; see Proposition \ref{prop:recursion mu} below. Before proving that result, we note the following corollary of Lemma \ref{lem:Sigma n}.

\begin{corollary} The matrix ${\bf A}_n(1)$ is a Markov matrix for all $n\geqslant 1$.
\end{corollary}

For a Dirac measure $\delta_{x}$, we let $(\delta_x)^r=\delta_{rx}$ be its $r$-fold convolution product with itself. This allows one to define the value of $p(\delta_x)$, for any polynomial $p\in\mathbb{C}[z]$.  

\begin{proposition}\label{prop:recursion mu}
For $n\geqslant 1$, one has 
\begin{equation}\label{eq:recursion mu}
\boldsymbol{\mu}_n={\bf A}_n\big(\delta_{1/k^{n}(k-1)}\big)\, \boldsymbol{\mu}_{n-1}
\end{equation}
where ${\bf A}_n(z)$ is the matrix-valued function defined in Eq.~\eqref{eq:A matrix}. 
\end{proposition}

\begin{proof}
Fix an $i\in\{1,\ldots,d\}$ and consider the $i$th entry of the vector on the right-hand side of \eqref{eq:recursion mu}. For this entry, we have 
\begin{align*}
e_{i}^T{\bf A}_n\big(&\delta_{1/k^{n}(k-1)}\big)\, \boldsymbol{\mu}_{n-1}=\sum_{j=1}^d \frac{\varSigma_j(n-1)\, b_{ij}(\delta_{1/k^{n}(k-1)})}{\varSigma_i(n)} \mu_{n-1,j}\\
&=\frac{1}{\varSigma_i(n)}\sum_{j=1}^d  \sum_{\ell=0}^{k^{n}-k^{n-1}-1} b_{ij}(\delta_{1/k^{n}(k-1)}) f_{j}(k^{n-1}+\ell)\,\delta_{\ell/k^{n-1}(k-1)}\\
&=\frac{1}{\varSigma_i(n)}  \sum_{\ell=0}^{k^{n}-k^{n-1}-1}\sum_{a=0}^{k-1} \sum_{j=1}^d ({\bf B}_a)_{ij}\, (\delta_{a/k^{n}(k-1)}) f_{j}(k^{n-1}+\ell)\,\delta_{\ell/k^{n-1}(k-1)}\\
&=\frac{1}{\varSigma_i(n)}  \sum_{\ell=0}^{k^{n}-k^{n-1}-1}\sum_{a=0}^{k-1} \delta_{(k\ell+a)/k^{n}(k-1)} \sum_{j=1}^d ({\bf B}_a)_{ij} \, f_{j}(k^{n-1}+\ell)\\
&=\frac{1}{\varSigma_i(n)}  \sum_{\ell=0}^{k^{n}-k^{n-1}-1}\sum_{a=0}^{k-1} \delta_{(k\ell+a)/k^{n}(k-1)}\, f_{i}(k^{n}+k\ell+a)\\
&=\frac{1}{\varSigma_i(n)}  \sum_{m=0}^{k^{n+1}-k^{n}-1} f_{i}(k^{n}+m)\,\delta_{m/k^{n}(k-1)}\\
&=\mu_{n,i}=e_{i}^T \boldsymbol{\mu}_{n}.
\end{align*}
Here, the third equality follows using the definition of $b_{ij}(z)$, the fifth equality follows by invoking \eqref{eq:k kernel} and the sixth step is just a change of index.  
\end{proof}

The following is an immediate consequence of Proposition~\ref{prop:recursion mu}.

\begin{corollary}\label{coro:recursion mu hat} 
For $n\geqslant 1$, the Fourier coefficients $\widehat{\mu}_n(t)$ satisfy \begin{equation}\label{eq: recursion mu hat} \widehat{\boldsymbol{\mu}_n}(t)= {\bf A}_n\big(\ee^{-\frac{2\pi \ii t}{k^n(k-1)}}\big)\, \widehat{\boldsymbol{\mu}_{n-1}}(t).
\end{equation}
for all $t\in\mathbb{Z}$.
\end{corollary}

\begin{remark}
While the convergents $\mu_n$ are pure point probability measures on $\mathbb{T}$, one can also consider the measures $\nu_n=\delta_\mathbb{Z}\ast \mu_n$, which are $\mathbb{Z}$-periodic measures in $\mathbb{R}$. Here, one has $\widehat{\nu}_n\colon \mathbb{R}\to \mathbb{C}$, with $\widehat{\mathbb{R}}=\mathbb{R}$. In the discussion below, we only carry out the analysis for $\mu_n$, but the (vague) convergence of the relevant matrix products also hold for $\nu_n$. \exend
\end{remark}

Using Equation~\eqref{eq:recursion mu}, one can construct the infinite matrix convolution and hope for the existence of the limit vector
\begin{equation}\label{eq: limit measure}
\boldsymbol{\mu}:=\bigg(\bigast_{n=1}^{\infty} {\bf A}_{n} \big(\delta_{1/k^n(k-1)}\big)\bigg)\,\boldsymbol{\mu}_0.
\end{equation}
The existence of the limit vector $\boldsymbol{\mu}$ depends on the compact convergence of the Fourier coefficients, which boils down to the compact convergence of the analytic matrix product 
\begin{equation*}
\prod_{n=1}^{\infty}  {\bf A}_n\big(\ee^{-\frac{2\pi \ii t}{k^n(k-1)}}\big), 
\end{equation*} a property we demonstrate below. 

We require the following lemma that the positivity of ${\bf B}$ implies the convergence of the quotients $\varSigma_i(n-1)/\varSigma_j(n)$ to a positive value. 

\begin{lemma}\label{lem: primitive quotients} Let $f$ be a primitive $k$-regular sequence.  Then for each $i,j\in\{1,\ldots,d\}$, the limit of the quotient $\varSigma_j(n-1)/\varSigma_i(n)$ exists and is non-zero.
\end{lemma}

\begin{proof} This follows from the fact that ${\bf B}$ is positive and that $f$ is non-negative (and not trivial), so that given an $i\in\{1,\ldots,d\}$ there is a constant $c_i>0$ such that \begin{equation}\label{cisigma}\varSigma_i(n)\sim c_i\cdot\rho_{_{\rm PF}}({\bf B})^n,\end{equation} as $n\to \infty$, where the positive real number $\rho_{_{\rm PF}}({\bf B})$ is the Perron--Frobenius eigenvalue of ${\bf B}$. To see this, let $\rho_{_{\rm PF}}({\bf B})$ be the Perron--Frobenius eigenvalue of ${\bf B}$ and let ${\bf v}$ be the corresponding positive Perron--Frobenius eigenvector. Let $j$ be a positive integer so that ${\bf B}^j {\bf w}$ is entry-wise greater than ${\bf v}$. Then for $n>j$, \begin{multline*}\varSigma_i(n)=e_i^T {\bf B}^n\, {\bf w}=e_i^T {\bf B}^{n-j}\, {\bf B}^{j}\, {\bf w}=e_i^T {\bf B}^{n-j}\, {\bf v}+e_i^T {\bf B}^{n-j}\,  ({\bf B}^{j}\,{\bf w}-{\bf v})\\ \geqslant e_i^T {\bf B}^{n-j}\, {\bf v}=e_i^T \rho_{_{\rm PF}}({\bf B})^{n-j}\, {\bf v}=(e_i^T{\bf v}/\rho_{_{\rm PF}}({\bf B})^j)\cdot \rho_{_{\rm PF}}({\bf B})^n.\end{multline*} Since $\rho_{_{\rm PF}}({\bf B})$ is a simple eigenvalue of ${\bf B}$ with maximal modulus, comparing the eigenvalue expansion of the linear recurrent sequence $\varSigma_i(n)$ with this inequality proves \eqref{cisigma}. The lemma follows immediately.
\end{proof}

For the proof of the following proposition, we adapt a technique used by Baake and Grimm in determining the intensities of Bragg peaks of Pisot substitutions via the internal Fourier cocycle \cite[Thm.~4.6]{BG2019}. 

\begin{proposition}\label{prop:conv general}
Let $f$ be a primitive $k$-regular sequence. Then the matrix product $$\prod_{n=1}^{\infty} {\bf A}_n\big(\ee^{-\frac{2\pi \ii t}{k^n(k-1)}}\big),$$ as a function of $t$, converges uniformly on $\mathbb{T}$ and compactly on $\mathbb{R}$.  
\end{proposition}

\begin{proof} 
It follows from the convergence of the quotients that ${\bf A}_n(z)\to {\bf A}(z)$ as $n\to\infty$, with $({\bf A}(z))_{ij}=c_{ij}b_{ij}(z)$. In particular, one has ${\bf A}_n(1)\to {\bf A}(1)$, where every ${\bf A}_n(1)$ is a Markov matrix. 
Setting ${\bf S}_{n}(t):={\bf A}_n\big(\ee^{\frac{-2\pi it}{(k-1)}}\big)$, one gets the convergence 
$\bS_n(t)\to \bS(t):=\bA(\ee^{\frac{-2\pi it}{(k-1)}})$.  
One has
\begin{equation}\label{AmSm}{\bf A}^{(m)}\big(\ee^{\frac{-2\pi it}{(k-1)}}\big)={\bf S}^{(m)}(t):={\bf S}_m\bigg(\frac{t}{k^m}\bigg)\, {\bf S}_{m-1}\bigg(\frac{t}{k^{m-1}}\bigg)\cdots\, {\bf S}_1\bigg(\frac{t}{k}\bigg).\end{equation}
Note that ${\bf S}^{(m)}(0)={\bf S}_m(0)\, {\bf S}_{m-1}(0)\cdots{\bf S}_1(0)$ is also a Markov matrix. 

For $0\leqslant r\leqslant \ell$ define the product
\[
\bS^{(\ell,r)}(t):=\bS_{\ell}\bigg(
\frac{t}{k^\ell}\bigg)\cdots\, \bS_{r}\bigg(
\frac{t}{k^r}\bigg).
\]
These products satisfy the identity
$\bS^{(\ell,r)}(t)=\bS^{(\ell,c)}(t)\, \bS^{(c,r)}(t)$, 
for $0\leqslant r\leqslant c\leqslant \ell$.
Note that $\bS^{(\ell,0)}(t)=\bS^{(\ell)}(t)$. Moreover,   
$\bS^{(\ell,r)}(0)$ is a Markov matrix for any such $\ell,r$, being a finite product of Markov matrices. 
The first goal is to show that ${\bf S}^{(m,\ell)}(t)$ is equicontinuous at $t=0$ for all $\ell$ satisfying $0\leqslant\ell\leqslant m$. 

Employing a telescoping argument, one obtains the equality
\[
\bS^{(m,\ell)}(t)-\bS^{(m,\ell)}(0)=
\sum_{j=\ell-1}^{m-1} \left(\bS^{(m,j+2)}(0)\, \bS^{(j+1,\ell)}(t)-
\bS^{(m,j+1)}(0)\, \bS^{(j,\ell)}(t)\right)
\]
This implies the following 
\begin{align*}
\|\bS^{(m,\ell)}(t)&-\bS^{(m,\ell)}(0)\|_{\infty}\\ &\leqslant \sum_{j=0}^{m-1}
\|\bS^{(m,j+2)}(0)\,\bS^{(j+1,\ell)}(t)-
\bS^{(m,j+1)}(0)\, \bS^{(j,\ell)}(t)\|_{\infty}\\
&\leqslant \sum_{j=0}^{m-1} \|\bS^{(m,j+2)}(0)\|_{\infty}\cdot
\|\bS^{(j+1,\ell)}(t)-\bS_{j+1}(0)\, \bS^{(j,\ell)}(t) \|_{\infty}\\
&\leqslant \sum_{j=0}^{m-1} \|\bS^{(m,j+2)}(0)\|_{\infty}\cdot
\|\bS_{j+1}(\tfrac{t}{k^{j+1}})-\bS_{j+1}(0)\|_{\infty}\cdot \|\bS^{(j,\ell)}(t)\|_{\infty}\\
&\leqslant \sum_{j=0}^{m-1} 
\|\bS_{j+1}(\tfrac{t}{k^{j+1}})-\bS_{j+1}(0)\|_{\infty},
\end{align*}
where the last step follows from the properties $\|\bS^{(m,j+2)}(0)\|_{\infty}=1$ and $$\|\bS^{(j,\ell)}(t)\|_{\infty}\leqslant \|\bS^{(j,\ell)}(0)\|_{\infty}=1,$$ since the matrices ${\bf B}_a$ are non-negative and since both $\bS^{(j,\ell)}(t)$ and $\bS^{(j,\ell)}(0)$ are Markov. 

Now, let $\varepsilon>0$ be given and choose $\delta>0$ such that
\[\|\bS_{j+1}(\tfrac{t}{k^{j+1}})-\bS_{j+1}(0)\|_{\infty}<\frac{\varepsilon}{k^{j+1}}\]
holds whenever $|\tfrac{t}{k^{j+1}}|<\delta$. This yields
\begin{equation}\label{eq:equicont general}
\|\bS^{(m,\ell)}(t)-\bS^{(m,\ell)}(0)\|_{\infty} 
 \leqslant \sum_{j=0}^{m-1} 
\|\bS_{j+1}(\tfrac{t}{k^{j+1}})-\bS_{j+1}(0)\|_{\infty}
< \sum_{j=0}^{\infty} \frac{\varepsilon}{k^{j}} 
=\varepsilon\big(\frac{k}{k-1}\big)
\end{equation} 
which proves that $\bS^{(m,\ell)}(t)$ is equicontinuous at $t=0$.

To show compact convergence, we prove that $\bS^{(m)}(t)=\bS^{(m,0)}(t)$ is uniformly Cauchy for any given compact set $K\subseteq \mathbb{Z}_{>0}$. To this end, choose $p$ such that $|\tfrac{t}{k^p}|<\delta$ for all $t\in K$, which implies $|\tfrac{t}{k^j}|<\delta$ for all $j\geqslant p$ and $t\in K$.  One then has
\begin{align*}
\|\bS^{(p+q+r)}(t)-\bS^{(p+q)}(t)\|_{\infty} & \leqslant 
\|\bS^{(p+q+r,p+1)}(t)-\bS^{(p+q,p+1)}(t)\|_{\infty}\cdot\|\bS^{(p)}(t)\|_{\infty}\\
& \leqslant \|\bS^{(p+q+r,p+1)}(t)-\bS^{(p+q,p+1)}(t)\|_{\infty},
\end{align*}
since $\|\bS^{(p)}(t)\|\leqslant 1$. Thus, using the triangle inequality, we obtain
\begin{multline*}
\|\bS^{(p+q+r,p+1)}(t)-\bS^{(p+q,p+1)}(t)\|_{\infty}
\leqslant\|\bS^{(p+q+r,p+1)}(t)-\bS^{(p+q+r,p+1)}(0)\|_{\infty}\\
+
\|\bS^{(p+q+r,p+1)}(0)-\bS^{(p+q,p+1)}(0)\|_{\infty}+
\|\bS^{(p+q,p+1)}(t)-\bS^{(p+q,p+1)}(0)\|_{\infty}
\end{multline*}
where the first and the third summands are strictly less than $\varepsilon\cdot k/(k-1)$
by invoking Eq.~\eqref{eq:equicont general} for large enough $p$. 
The second summand splits further, into 
\begin{multline}
\|\bS^{(p+q+r,p+1)}(0)-\bS^{(p+q,p+1)}(0)\|_{\infty}\leqslant
\|\bS^{(p+q+r,p+1)}(0)-\bA(1)^{q+r}\|_{\infty}\\+\|\bA(1)^{q+r}-\bA(1)^{q}\|_{\infty}+\|\bS^{(p+q,p+1)}(0)-\bA(1)^q\|_{\infty}, \label{eq: sum S(0)}
\end{multline}
Since $\bS_{\ell}(0)$ converges to $\bA(1)$, one can choose $p^{\prime}$ large enough such that
\[\|\bS_{\ell}(0)-\bA(1)\|_{\infty}<\frac{\varepsilon}{k^\ell}\] for all $\ell\geqslant p^{\prime}$. One can then invoke a similar argument to show that the first and the third summands in Eq.~\eqref{eq: sum S(0)} are strictly less than $\varepsilon\cdot k/(k-1)$ whenever $p\geqslant p^{\prime}$. Finally, since $\bA(1)$ is Markov, $\bA(1)$ converges to a steady state matrix whence one can choose $q$ large enough such that 
the second summand is less than $\varepsilon$. 

Since $r$ is arbitrary, this means for all $m,n\geqslant \max\left\{p,p^{\prime}\right\}+q$ one has that
\[
\|\bS^{(m)}(t)-\bS^{(n)}(t)\|_{\infty}<\varepsilon\left(\frac{5k-1}{k-1}\right),
\]
and thus, $\bS^{(n)}(t)$ converges uniformly on $\mathbb{T}$ and compactly on $\mathbb{R}$.
\end{proof}

The following result is now an immediate corollary  of Proposition \ref{prop:conv general} by L\'evy's continuity theorem \cite[Thm.~3.14]{BFbook}.

\begin{theorem}\label{thm:muexists}
Let $f$ be a primitive $k$-regular sequence. Then the weak limit measure vector $\boldsymbol \mu$ in Eq.~\eqref{eq: limit measure} exists. 
\end{theorem}

\begin{remark}
As stated previously, the non-negativity assumptions on the $k$ matrices ${\bf B}_a$ and the positivity (primitivity, initially) assumption on ${\bf B}$ are natural, especially if one views it as the corresponding analogue of the substitution matrix ${\bf M}_\varrho$ for shift spaces arising from a substitution $\varrho$. For the latter, the primitivity of ${\bf M}_\varrho$ implies that the hull $\mathbb{X}_{\varrho}$ defined by $\varrho$ is strictly ergodic, and hence the diffraction measure $\widehat{\gamma}_{w}$ is constant on $\mathbb{X}_{\varrho}$ given an arbitrary choice of weight function $w$; see Baake's and Grimm's monograph \cite{BGbook} for details and definitions. 
Alternatively, one can consider matrix Riesz products of measures, where the approximants are absolutely continuous with respect to Lebesgue measure, and weak convergence to the matrix of correlation measures is likewise guaranteed by primitivity; see Bufetov and Solomyak~\cite[Lem.~2.2]{BS2014} and Queff\'elec \cite[Thm.~8.1]{Qbook}. \exend
\end{remark}

In what follows, we recover an analogous uniqueness result for the basis $\left\{f_1,\ldots,f_d\right\}$ of the $\mathbb{Q}$-vector space $\mathcal{V}_k(f)$.  

\begin{proposition}\label{prop: uniqueness}
Let $f$ be a primitive $k$-regular sequence. Then $\boldsymbol{\mu}=(\mu_f,\ldots,\mu_f)^{T}$ where $\mu_f$ is a probability measure on $\mathbb{T}$. That is, for each $i\in\{1,\ldots,d\}$, the weak limit of $\mu_{n,i}$ is $\mu_f$.
\end{proposition}

\begin{proof} For a fixed $t\in\mathbb{R}$, one has $\widehat{\boldsymbol{\mu}}(t)={\bf L}(t)\, \widehat{\boldsymbol{\mu}}_0(t)$, where $${\bf L}(t):=\prod_{n=1}^{\infty}  {\bf A}_n\big(\ee^{-\frac{2\pi \ii t}{k^n(k-1)}}\big).$$ From the convergence in Lemma~\ref{lem: primitive quotients}, ${\bf A}_{n}(1)\to {\bf A}(1)$ as $n\to \infty$, where ${\bf A}(1)$ is a primitive Markov matrix, where the primitivity of ${\bf A}(1)$ follows from the positivity of ${\bf B}$ and Lemma~\ref{lem: primitive quotients}.
This means ${\bf A}(1)^n$ converges to the rank-$1$ projector $ {\bf P}_{\bf A}$ corresponding to the eigenvector ${\bf 1}=(1,\ldots,1)^{T}$. One then gets the equality 
$${\bf A}(1)\, {\bf L}(t)={\bf L}(t)= {\bf P}_{\bf A}\, {\bf L}(t),$$ which implies $\widehat{\boldsymbol{\mu}}(t)={\bf P}_{\bf A}\, \widehat{\boldsymbol{\mu}}(t)=c(t)\, {\bf 1}$. 
This means that the limit measures $\mu_i$ have the same Fourier coefficients for all $t$, and hence must correspond to the same measure $\mu=\mu_f$ for all $i$ satisfying $1\leqslant i\leqslant d$.  
\end{proof}

Theorem~\ref{main1} follows by combining Theorem~\ref{thm:muexists} and Proposition \ref{prop: uniqueness}.

\section{A natural probability measure associated with $f$}\label{sec:f}

In this section we will prove Theorem \ref{main2}, via the general situation described at the end of the Introduction. To set up, we recall the following notation, now for only a single $k$-regular sequence $f$, $$\varSigma_f(n):=\sum_{m=k^n}^{k^{n+1}-1}f(m)$$ and $$
   \mu^{}_{f,n} \, := \, \frac{1}{\varSigma_f(n)} \sum_{m=0}^{k^{n+1}-k^n - 1}
   f(k^n + m) \, \delta^{}_{m / k^n(k-1)},
$$
where $\delta_x$ denotes the unit Dirac measure at $x$. Also as previously, let $f$ be defined by the matrices ${\bf B}_0,\ldots,{\bf B}_{k-1}$ and the vector ${\bf w}\in\mathbb{Z}^{d\times 1}$ such that $f(m)=e_1^T {\bf B}_{(m)_k} {\bf w}$, where $(m)_k=i_s\cdots i_1i_0$ is the base-$k$ expansion of $m$ and ${\bf B}_{(m)_k}:={\bf B}_{i_0}{\bf B}_{i_1}\cdots {\bf B}_{i_s}.$ As before, set ${\bf B}:=\sum_{a=0}^{k-1}{\bf B}_a.$ Let $\rho({\bf M})$ denote the spectral radius of the matrix ${\bf M}$ and denote the {\em joint spectral radius} of a finite set of matrices $\{{\bf M}_1,{\bf M}_2,\ldots, {\bf M}_{\ell}\}$, by the real number $$\rho^*(\{{\bf M}_1,{\bf M}_2,\ldots, {\bf M}_{\ell}\})=\limsup_{n\to\infty}\max_{1\leqslant i_0,i_1,\ldots,i_{n-1}\leqslant \ell}\left\| {\bf M}_{i_0}{\bf M}_{i_1}\cdots{\bf M}_{i_{n-1}}\right\|^{1/n},$$
where $\|\cdot\|$ is any (submultiplicative) matrix norm. This quantity was introduced by Rota and Strang \cite{RS1960} and has a wide range of applications. For an extensive treatment, see Jungers's monograph \cite{J2009}.

Unlike Theorem \ref{main1}, so also unlike the previous section, we do not yet assume that $f$ is non-negative, nor that ${\bf B}$ is positive. However, to avoid degeneracies (discussed in the next section), we assume that the spectral radius $\rho({\bf B})$ is the unique dominant eigenvalue of ${\bf B}$, that $$\rho:=\rho({\bf B})>\rho^*(\{{\bf B}_0,\ldots,{\bf B}_{k-1}\})=:\rho^*,$$ that for $n$ large enough $\varSigma_f(n)\neq 0$ and that the asymptotical behaviour of $\varSigma_f(n)$ is determined by $\rho({\bf B})$. This last assumption will be highlighted and made explicit below. 

To exploit the asymptotical behaviour of $\varSigma_f(n)$, we use a result of Dumas \cite[Thm.~3]{D2013} on the asymptotic nature of the partial sums $\sum_{m\leqslant x}f(m)$. Throughout this paper, we have used the convention that ${\bf B}_{(m)_k}:={\bf B}_{i_0}{\bf B}_{i_1}\cdots {\bf B}_{i_s}$, where $(m)_k=i_s\cdots i_1i_0$ is the base-$k$ expansion of $m$, however, the result of Dumas \cite{D2013} that we use here requires multiplying in the opposite order. These two representations are related via matrix transposition; $$f(m)=e_1^T {\bf B}_{(m)_k} {\bf w}={\bf w}^T {\bf B}^{T}_{(m)_k} e_{1}.$$ To set up Dumas' result, we require some further notation. Let ${\bf J}_\rho=\rho{\bf I}_v+{\bf Z}$ be the Jordan block associated with $\rho$ from the Jordan normal form of ${\bf B}^{T}$, where ${\bf I}_v$ is the $v\times v$ identity matrix and ${\bf Z}$ is the nilpotent matrix of index $v$ with ones on the superdiagonal and zeros elsewhere.  Let ${\bf V}_\rho$ be the $d\times v$ matrix whose columns are the elements associated with $\rho$ from some Jordan basis of ${\bf B}^{T}$. Let the vector ${\bf e}_\rho$ be such that ${\bf V}_\rho {\bf e}_\rho$ equals the component of $e_{1}$ in the invariant subspace of ${\bf B}$ associated with $\rho$. Finally, we define the matrix-valued function ${\bf F}_\rho:\mathbb{R}\to\mathbb{C}^{d\times v}$ by \begin{equation}\label{dilation}{\bf F}_{\rho}(x) \cdot {\bf J}_{\rho} = \sum_{a=0}^{k-1}{\bf B}_{a}^{T} \cdot {\bf F}_{\rho}(k x-a),\end{equation} with the boundary conditions $${\bf F}_{\rho}(x)=\begin{cases}0 & \text{ for } x \leqslant 0\\ {\bf V}_{\rho} &\text{ for } x \geqslant 1.\end{cases}$$ The function ${\bf F}_\rho$ exists and is unique since $\rho>\rho^*$. Moreover, the function ${\bf F}_\rho$ is H\"older continuous with exponent $\alpha$ for any $\alpha<\log_k(\rho/\rho^*).$ Functional equations such as \eqref{dilation} are known as \emph{dilation equations} or \emph{two-scale difference equations} in the literature; seminal work on these was done by Daubechies and Lagarias \cite{DL1991,DL1992}. See also Micchelli and Prautzsch \cite{MP1989}. A point of interest for our context is that both the above-mentioned papers of Daubechies and Lagarias as well as the seminal paper of Allouche and Shallit \cite{AS1992} introducing $k$-regular sequences were published within the span of one year. Twenty years later, Dumas \cite{D2013}---extending ideas of Coquet \cite{C1983}---connected these ideas by showing explicitly how one can use a dilation equation to determine the  asymptotic growth of the partial sums of a regular sequence. We record his result here in the special case fit for our purpose.

\begin{theorem}[Dumas]\label{thm:Dumas} Suppose that the spectral radius $\rho=\rho({\bf B})$ is the unique dominant eigenvalue of ${\bf B}$ and that $\rho({\bf B})>\rho^*(\{{\bf B}_0,\ldots,{\bf B}_{k-1}\})$. Then $$\sum_{m\leqslant x}f(m) = {\bf w}^T\, {\bf E}_\rho(\log_k(x))\, {\bf e}_\rho + O(x^{\log_k (r)}),$$ where $r$ is any positive number strictly between $\rho$ and the modulus of the next-largest eigenvalue of ${\bf B}$ and \begin{multline*}{\bf E}_\rho(x):=({\bf I}_d-{\bf B}_{0}^{T})\, {\bf V}_\rho\, ({\bf I}_v-{\bf J}_\rho)^{-1}\\+\big(-({\bf I}_d-{\bf B}_{0}^{T})\, {\bf V}_\rho\, ({\bf I}_v-{\bf J}_\rho)^{-1}+{\bf F}_\rho(k^{\{ x\}-1})\big)\cdot\rho^{\lfloor x\rfloor+1}\cdot {\bf P}_\rho(\lfloor x\rfloor),\end{multline*} where ${\bf P}_\rho(\lfloor x\rfloor)=({\bf I}_v+(1/\rho){\bf Z})^{\lfloor x\rfloor}$. \end{theorem}

A point to be made here is that, while Theorem \ref{thm:Dumas} is certainly technical, in the cases we apply it, for integers $x\in[k^n, k^{n+1})$, the integer part $\lfloor x\rfloor=n$ is constant, so only the dependence on the fractional parts $\{x\}$ will need to be dealt with. For a detailed example of how this theorem  can be applied to give a distribution function, see Baake and Coons~\cite[Sec.~3]{BC2018}, where they give an account concerning the Stern sequence. 

Now, applying Theorem \ref{thm:Dumas} to the complete sums $\varSigma_f(n)$ and using the transposed representation as that theorem requires, gives $$\varSigma_f(n)=\sum_{m=k^n}^{k^{n+1}-1}f(m)
=\rho^{n+1} {\bf w}^T \left({\bf F}_\rho\bigg(\frac{k-1/k^n}{k}\bigg)-{\bf F}_\rho\bigg(\frac{1}{k}\bigg)\right){\bf P}_\rho(n)\, {\bf e}_\rho+O(r^n),$$ for some $r<\rho$. As $n\to\infty$, since ${\bf F}_\rho$ is H\"older continuous, $${\bf F}_\rho\bigg(\frac{k-1/k^n}{k}\bigg)={\bf V}_\rho+o(1),$$ so that \begin{equation}\label{eq:denom}\varSigma_f(n)
=\rho^{n+1} {\bf w}^T \left({\bf V}_\rho-{\bf F}_\rho\bigg(\frac{1}{k}\bigg)\right){\bf P}_\rho(n)\, {\bf e}_\rho+o(\rho^n).\end{equation}
For $o(\rho^n)$ to be the true error term in Eq.~\eqref{eq:denom}, for sufficiently large $n$ we must have that \begin{equation}\label{condition} {\bf w}^T \left({\bf V}_\rho-{\bf F}_\rho\bigg(\frac{1}{k}\bigg)\right){\bf P}_\rho(n)\, {\bf e}_\rho \neq 0.\end{equation} Since $${\bf P}_\rho(n)=({\bf I}_v+(1/\rho){\bf Z})^{n}=\sum_{j=0}^{v-1}{n\choose j}\rho^{j}{\bf Z}^j,$$ the non-zero condition \eqref{condition} holds precisely when there is a $j\in\{0,\ldots,v-1\}$ such that \begin{equation}\label{zjcond} {\bf w}^T \left({\bf V}_\rho-{\bf F}_\rho\bigg(\frac{1}{k}\bigg)\right){\bf Z}^j\, {\bf e}_\rho \neq 0.\end{equation} This being the case, if $\ell$ is the largest $j\in\{0,\ldots,v\}$ for which \eqref{zjcond} holds, then \begin{equation}\label{sigfnl} \varSigma_f(n)
=\rho^{n-\ell+1}{n\choose \ell} {\bf w}^T \left({\bf V}_\rho-{\bf F}_\rho\bigg(\frac{1}{k}\bigg)\right){\bf Z}^\ell\, {\bf e}_\rho+o(\rho^nn^\ell).\end{equation}

Arguing as in the previous paragraph allows us to prove the following result. 

\begin{theorem}\label{main2explicit} Let $f$ be an integer-valued $k$-regular sequence. Suppose that $\rho({\bf B})$ is the unique dominant eigenvalue of ${\bf B}$, $\rho({\bf B})>\rho^*(\{{\bf B}_0,\ldots,{\bf B}_{k-1}\})$ and that Eq.~\eqref{zjcond} holds. Then, the limit $\mu_{f}([0,x])$ of the sequence $\mu_{f,n}([0,x])$ exists. Moreover, the function $\mu_{f}([0,x])$ is H\"older continuous with exponent $\alpha$ for any $\alpha<\log_k(\rho/\rho^*)$.
\end{theorem}

\begin{proof} Let $x\in\mathbb{T}$ and consider the sequence of functions $\mu_{f,n}([0,x])$. Then, applying the argument of the above paragraph with $(k-1/k^n)/k$ replaced by $1+(k-(1+1/k^n))x$, we have \begin{multline*}\varSigma_f(n)\cdot \mu_{f,n}([0,x])=\sum_{m=k^n}^{k^n(1+(k-(1+1/k^n))x)}f(m)\\
=\rho^{n+1} {\bf w}^T \left({\bf F}_\rho\bigg(\frac{1+(k-(1+1/k^n))x}{k}\bigg)\right.\\
\left.-{\bf F}_\rho\bigg(\frac{1}{k}\bigg)\right){\bf P}_\rho(n)\, {\bf e}_\rho+o(\rho^n),
\end{multline*} for some $r<\rho$. As before, using the H\"older continuity of ${\bf F}_\rho$ and letting $\ell$ be the largest value for which \eqref{zjcond} holds, we obtain that $\varSigma_f(n)\cdot \mu_{f,n}([0,x])$ equals $$
\rho^{n-\ell+1} {n\choose \ell}{\bf w}^T \left({\bf F}_\rho\bigg(\frac{1+(k-1)x}{k}\bigg)-{\bf F}_\rho\bigg(\frac{1}{k}\bigg)+o(1)\right){\bf Z}^\ell\, {\bf e}_\rho+o(\rho^n n^\ell).
$$ Using the asymptotic \eqref{sigfnl} yields \begin{equation*} \mu_{f,n}([0,x])=\frac{{\bf w}^T \left({\bf F}_\rho\big(\frac{1+(k-1)x}{k}\big)-{\bf F}_\rho\big(\frac{1}{k}\big)+o(1)\right){\bf Z}^\ell\, {\bf e}_\rho+o(1)}{{\bf w}^T \left({\bf V}_\rho-{\bf F}_\rho\big(\frac{1}{k}\big)\right){\bf Z}^\ell\, {\bf e}_\rho+o(1)}.\end{equation*} By our assumption that \eqref{zjcond} holds, the denominator limits to a nonzero value, and so the point-wise limit of $\mu_{f,n}([0,x])$ exists for all $x$; explicitly, \begin{equation*} \mu_{f}([0,x])=\lim_{n\to\infty}\mu_{f,n}([0,x])=\frac{{\bf w}^T \left({\bf F}_\rho\big(\frac{1+(k-1)x}{k}\big)-{\bf F}_\rho\big(\frac{1}{k}\big)\right){\bf Z}^\ell\, {\bf e}_\rho}{{\bf w}^T \left({\bf V}_\rho-{\bf F}_\rho\big(\frac{1}{k}\big)\right){\bf Z}^\ell\, {\bf e}_\rho}.\end{equation*} Finally, we note that the function $\mu_{f}([0,x])$ is H\"older continuous with exponent $\alpha$ for any $\alpha<\log_k(\rho/\rho^*)$, a property it inherits directly from ${\bf F}_\rho$. 
\end{proof}

While we have used the notation $\mu_{f}([0,x])$, the assumptions of Theorem \ref{main2explicit} are not strong enough to guarantee the existence of a measure $\mu_f$ for which $\mu_{f}([0,x])$ is a distribution function. The following result contains a sufficient condition and is a generalisation of the result stated in the Introduction as Theorem \ref{main2}. 

\begin{theorem}\label{BVexplicit} Let $f$ be an integer-valued $k$-regular sequence. Suppose that $\rho({\bf B})$ is the unique dominant eigenvalue of ${\bf B}$, $\rho({\bf B})>\rho^*(\{{\bf B}_0,\ldots,{\bf B}_{k-1}\})$ and that Eq.~\eqref{zjcond} holds. Suppose that $\mu_{f}([0,x])$, as provided by Theorem \ref{main2explicit}, is a function of bounded variation. Then $\mu_{f}([0,x])$ is the distribution function of a measure $\mu_f$, which is continuous with respect to Lebesgue measure.
\end{theorem}

\begin{proof} We start with the function $\mu_{f}([0,x])$ provided by Theorem \ref{main2explicit}. Assuming that $\mu_{f}([0,x])$ is of bounded variation, we form the (possibly) signed measure $\mu_f$ via a Riemann--Stieltjes integral, assigning the value $$\int_a^b {\rm d} \mu_f([0,x]):=\mu_{f}([0,b])-\mu_{f}([0,a])$$ to any interval $(a,b]\subseteq\mathbb{T}$; see Lang \cite[Chp.~X]{Lang} for details on Riemann--Stieltjes integration and measure. This is a Borel measure and has distribution function $\mu_{f}([0,x])$. Since the distribution functions $\mu_{f,n}([0,x])$ of the measures $\mu_{f,n}$ are converging point-wise to the continuous distribution function $\mu_{f}([0,x])$ of the measure $\mu_{f}$, the measures $\mu_{f,n}$ are converging weakly to $\mu_{f}$. Moreover, the measure $\mu_{f}$ has no pure points since $\mu_{f}([0,x])$ is continuous, thus $\mu_{f}$ is continuous with respect to Lebesgue measure.
\end{proof}

This result gives Theorem \ref{main2} as a corollary.

\begin{proof}[Proof of Theorem \ref{main2}] This follows immediately from Theorem \ref{main2explicit} using the fact that the partial sums $\sum_{m\leqslant y}f(m)$ are  increasing with $y$ to show that $\mu_f([0,x])$ is increasing, so of bounded variation.
\end{proof}

We note that the H\"older exponent of the distribution function $\mu_f([0,x])$ is related to (possibly fractal) geometric properties of the measure $\mu_f$. To formalise this, let $\mu$ be a probability measure on $\mathbb{T}$. The \emph{lower local dimension} of $\mu$ at $x\in\mathbb{T}$ is given by 
\[
\underline{{\rm dim}}_{\rm loc} \, \mu(x)\colon=\liminf_{r\to 0^+} \frac{\log\mu(B_r(x))}{\log r}, 
\]
where as usual $B_r(x)$ denotes the ball of radius $r>0$ with centre $x$. One can also view $\underline{{\rm dim}}_{\rm loc} \, \mu(x)$ as the ``best possible local H\"older exponent,'' that is, 
\[
\underline{{\rm dim}}_{\rm loc} \, \mu(x)=\sup \left\{\alpha\geqslant 0\colon \mu\big(B_r(x)\big)=O(r^{\alpha})\ \mbox{as}\ r\to 0^+\right\}, 
\]
see Bufetov and Solomyak \cite[Sec.~3.2]{BS2018}, Baake et al.~\cite{BGKS2019} and Mattila, Mor\'an and Rey~\cite{MMR2000}. If $\mu$ is absolutely continuous, one has $\underline{{\rm dim}}_{\rm loc} \, \mu(x)=1$ for Lebesgue almost all $x$, whereas $\underline{{\rm dim}}_{\rm loc} \, \mu(x)=0$ for all $x$ satisfying $\mu(\left\{x\right\})\neq 0$ when $\mu$ is pure point. 

\begin{proposition}
Let $\mu$ be a probability measure on $\mathbb{T}$ with distribution function $F$. If $F$ is H\"older continuous with exponent $\alpha$, then $\underline{{\rm dim}}_{\rm loc} \, \mu(x)\geqslant \alpha$ for all $x\in\mathbb{T}$. 
\end{proposition}

\begin{proof} 
One has, assuming $r<1$,
$$
\underline{{\rm dim}}_{\rm loc} \, \mu(x) =\liminf_{r\to 0^+}  \frac{\log\mu(B_r(x))}{\log r}= \liminf_{r\to 0^+}  \frac{\log|F(x+\frac{r}{2})-F(x-\frac{r}{2})|}{\log r}\geqslant \alpha, 
$$
where the inequality follows from the H\"older continuity of $F$.
\end{proof}

We have the following immediate corollary for positive $k$-regular sequences. 

\begin{corollary}
Let $f$ be a $k$-regular sequence which satisfies the conditions of Theorem~\ref{main2} and let $\mu_f$ be the corresponding continuous measure. Then, one has 
\[
\underline{{\rm dim}}_{\rm loc} (\mu_f,x)\geqslant \log_{k}\big(\rho/\rho^{\ast}\big)
\]
uniformly on $\mathbb{T}$. 
\end{corollary}

An interesting example is the Stern sequence \cite{BC2018}. Here one has a H\"older exponent strictly less than one. In this case    $\underline{{\rm dim}}_{\rm loc}(\mu_s,x)\geqslant \log_2({3}/{\tau})\approx 0.890721$ for all $x\in\mathbb{T}$ where $\tau$ is the golden mean.

\section{Some comments on assumptions}\label{nonprim}

Theorem \ref{main1} relies on one main assumption, the primitivity of $f$, and Theorem~\ref{BVexplicit} relies on a two key assumptions. In this section, we consider these assumptions, highlighting their necessity with examples. 

The assumption of primitivity of the $k$-regular sequence $f$ in Theorem \ref{main1} implies that the matrix ${\bf B}$ is positive; recall, this positivity was gained by starting with a primitive matrix and considering $f$ as a $k^j$-regular sequence, where $j$ was the minimal positive integer such that ${\bf B}^j$ is positive. Sometimes such a choice is not immediately possible. For example, consider the Josephus sequence $J$, which is $2$-regular and determined by $J(0)=0$ and the recursions $J(2n)=2J(n)-1$ and $J(2n+1)=2J(n)+1$. These recursions imply that the sequences $J$ and ${\bf 1}$ (the constant sequence) form a basis for $\mathcal{V}_2(J)$. With this basis, we arrive at the linear representation $J(n)=e_1^T\, {\bf B}_{(n)_2} (0,1)^T,$ where $${\bf B}_0 =\begin{pmatrix}
2 &-1 \\
0 & 1
\end{pmatrix}\quad\mbox{and}\quad {\bf B}_1 =\begin{pmatrix}
2 &1 \\
0 & 1
\end{pmatrix},\quad\mbox{so that}\quad{\bf B}=\begin{pmatrix}
4 &0 \\
0 & 2
\end{pmatrix}.$$ In this case, ${\bf B}$ is neither positive nor primitive, so Theorem \ref{main1} cannot be applied. However, according to Theorem \ref{main2}, one can apply the same limiting measures $\mu_{n,i}$ via the results in Section \ref{sec:f}. In the limit, one gets $\mu_{J}=h(x)\cdot{\lambda}$ whereas $\mu_{\bf 1}={\lambda}$, where $\lambda$ is normalised Haar measure on $\mathbb{T}$ and $h(x)=2x$ is the Radon--Nikodym density for the limit measure associated with
$J$. 

But, there is hope here---a change of basis allows the use of Theorem \ref{main1}. We need only note that if we conjugate ${\bf B}_0$ and ${\bf B}_1$ by the matrix $$\begin{pmatrix}
1 &-1 \\
1 & 1
\end{pmatrix},$$ we arrive at a new linear representation of the sequence $J$, and a new basis for $\mathcal{V}_2(J)$, where the matrix ${\bf B}$ is replaced by $$\begin{pmatrix}
1 &-1 \\
1 & 1
\end{pmatrix} {\bf B}\begin{pmatrix}
1 &-1 \\
1 & 1
\end{pmatrix}^{-1}=\begin{pmatrix}
1 &-1 \\
1 & 1
\end{pmatrix} \begin{pmatrix}
4 &0 \\
0 & 2
\end{pmatrix}\begin{pmatrix}
1 &-1 \\
1 & 1
\end{pmatrix}^{-1}=\begin{pmatrix}
3 &1 \\
1 & 3
\end{pmatrix},$$ which is positive. Thus, we can apply Theorem \ref{main1} using this new basis resulting in the vector of measures $\boldsymbol{\mu}_J=\mu_J(1,1)^T.$ This method has forgotten the measure $\mu_{\bf 1}$. This process, in general (and if possible), picks out the maximal measure, in the sense of growth of elements of $\mathcal{V}_k(f)$. This maximality becomes evident noticing that one can choose a basis for $\mathcal{V}_k(f)$ from ${\rm ker}_k(f)$; see Allouche and Shallit \cite[Thm.~2.2(b)]{AS1992}. For more information on the Josephus sequence and for a study on measures associated with affine $2$-regular sequences, see Evans \cite{Epre}. 

\begin{remark} The distribution function of $\mu_J$ is Lipschitz, so that $\underline{{\rm dim}}_{\rm loc}(\mu_J,x)\geqslant 1$ for all $x\in\mathbb{T}$. In particular, the local dimension exists and is one. \exend
\end{remark}

Shifting now to Theorem \ref{BVexplicit}, consider the assumption in that $\rho({\bf B})$ is the unique dominant eigenvalue of ${\bf B}$ (which holds if ${\bf B}$ were say, primitive). There are plenty of examples of regular sequences where this is not the case. Dumas \cite[Ex.~5]{D2013} gave an interesting example of a $2$-regular sequence, which we denote by $D$, whose associated matrix ${\bf B}$ has a negative integer entry, is not primitive and does not have a maximal eigenvalue. Dumas' sequence $D$ is defined  by setting $${\bf B}_0=\left(\begin{matrix} 1 & 0\\ 0& 1\end{matrix}\right)\quad\mbox{and}\quad {\bf B}_1=\left(\begin{matrix} 3 & -3\\ 3 &3\end{matrix}\right),$$ and \begin{equation}\label{D}D(m)=e_1^T\, {\bf B}_{(m)_2} e_1=e_1^T\, {\bf B}_{1}^{s_2(m)} e_1,\end{equation} where $s_2(m)$ is the sum of the bits of $m$. As in the previous section, we consider the sequence of pure point measures  \begin{equation}\label{muDn}\mu_{D,n}:=\frac{1}{\varSigma_D(n)}\sum_{m=2^{n}}^{2^{n+1}-1}D(m)\, \delta_{m/2^{n+1}},\end{equation} where \begin{equation}\label{SigmaDn}\varSigma_D(n):=\sum_{m=2^{n}}^{2^{n+1}-1}D(m)=\frac{3}{2}\cdot\big((1+\ii)(4+3\ii)^n+(1-\ii)(4-3\ii)^n\big),\end{equation} since in this case, the matrix ${\bf B}={\bf B}_0+{\bf B}_1$ has eigenvalues $4+3\ii$ and $4-3\ii$, each of modulus $5$. To show that the `limit' $\mu_D$ does not exist, it is enough to prove that the sequence $\big(\mu_{D,n}([0,1/2))\big)_{n\geqslant 0}$ does not have a limit. To see why this is enough, note that if the measure $\mu_{D}$ did exist, then so would its distribution function $\mu_D([0,x)):[0,1]\to\mathbb{R}$, and necessarily we would have that $\lim_{n\to\infty}\mu_{D,n}([0,1/2))=\mu_D([0,1/2))$.

\begin{proposition}\label{diverge} The sequence $\big(\mu_{D,n}([0,1/2))\big)_{n\geqslant 0}$ does not converge. Moreover, it is not eventually periodic.
\end{proposition}

\begin{proof} To prove this result, we show that there is a subsequence of this sequence which is unbounded. To this end, note that \eqref{muDn} gives $$\mu_{D,n}([0,1/2))=\frac{1}{\varSigma_D(n)}\sum_{m=2^{n}}^{2^{n}+2^{n-1}-1}D(m).$$ Using \eqref{D} we have that \begin{multline*}\sum_{m=2^{n}}^{2^{n}+2^{n-1}-1}D(m)=e_1^T\Bigg(\sum_{\substack{w\in\{0,1\}^*\\ |w|=n-1}}{\bf B}_w\Bigg){\bf B}_0{\bf B}_1 e_1\\ =e_1^T\left({\bf B}_0+{\bf B}_1\right)^{n-1}{\bf B}_1 e_1=\sum_{m=2^{n-1}}^{2^{n}-1}D(m)=\varSigma_D(n-1),\end{multline*} where we have used that ${\bf B}_0$ is the identity matrix to obtain the middle equality. Thus $$\mu_{D,n}([0,1/2))=\frac{\varSigma_D(n-1)}{\varSigma_D(n)}.$$ By a simple calculation applying \eqref{SigmaDn} to both the numerator and denominator, with some rearrangement, we have $$\mu_{D,n}([0,1/2))=\frac{1}{4+3\ii}\cdot\frac{1-\ee^{-\ii(2 \vartheta\cdot(n-1)-\pi/2)}}{1-\ee^{-\ii(2 \vartheta\cdot(n)-\pi/2)}},$$ where $\vartheta\approx 0.6435$ is the solution of $\cos\vartheta=4/5$. Now since $\vartheta$ is irrational and not a rational multiple of $\pi$, we have that the sequence of fractional parts $$\big(\{2 \vartheta\cdot(n)-\pi/2\}\big)_{n\geqslant 0}$$ is equidistributed in $[0,1)$. In particular, let $M>0$ be a positive integer. Then there is an $\varepsilon>0$ satisfying $$0<\varepsilon<\frac{\left|1-\ee^{-2\ii \vartheta}\right|}{5M+1},$$ and there are infinitely many $n$ such that $$\Big|1-\ee^{-\ii(2 \vartheta\cdot(n)-\pi/2)}\Big|<\varepsilon.$$ For these infinitely many $n$, $$\big|\mu_{D,n}([0,1/2))\big|>\frac{1}{5}\cdot\frac{\left|1-\ee^{-2\ii \vartheta}\right|-\varepsilon}{\varepsilon}=\frac{1}{5}\cdot\frac{\left|1-\ee^{-2\ii \vartheta}\right|}{\varepsilon}-\frac{1}{5}>M.$$ Since $M>0$ can be chosen arbitrarily large, $\big(\mu_{D,n}([0,1/2))\big)_{n\geqslant 0}$ is unbounded, which is the desired result.
\end{proof}

To observe large values of $\big(\mu_{D,n}([0,1/2))\big)_{n\geqslant 0}$ one must be patient. Figure \ref{fig:Dnoper} shows values of $\mu_{D,n}([0,1/2))$ for $n$ from $0$ to $100000$.
Modifying the proof of Proposition \ref{diverge}, \emph{mutatis mutandis}, one can show that the sequences $\big(\mu_{D,n}([0,1/2^k))\big)_{n\geqslant 0}$ are unbounded for every $k\geqslant 1$.

\begin{figure}[ht]
\begin{center}
  \includegraphics[width=.32\linewidth]{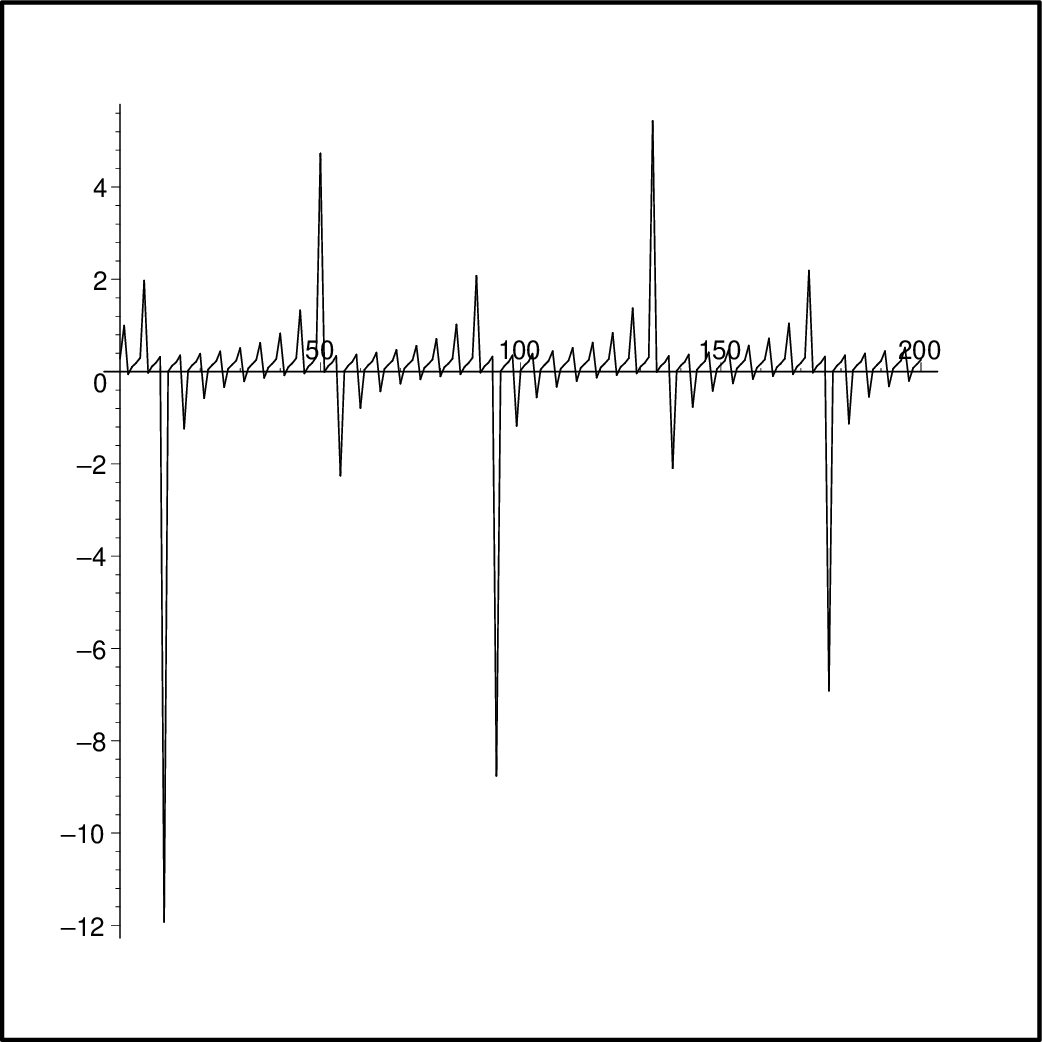}
  \includegraphics[width=.32\linewidth]{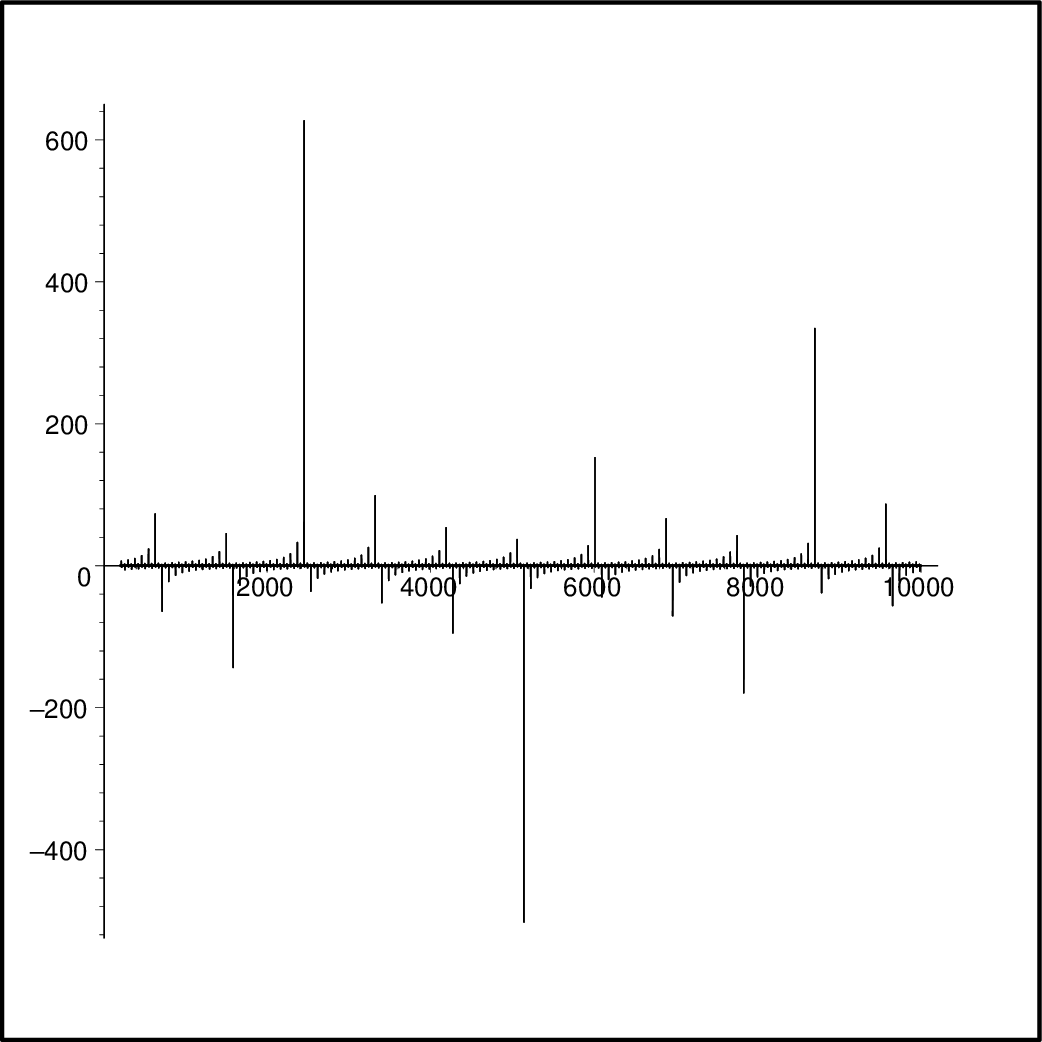}
  \includegraphics[width=.32\linewidth]{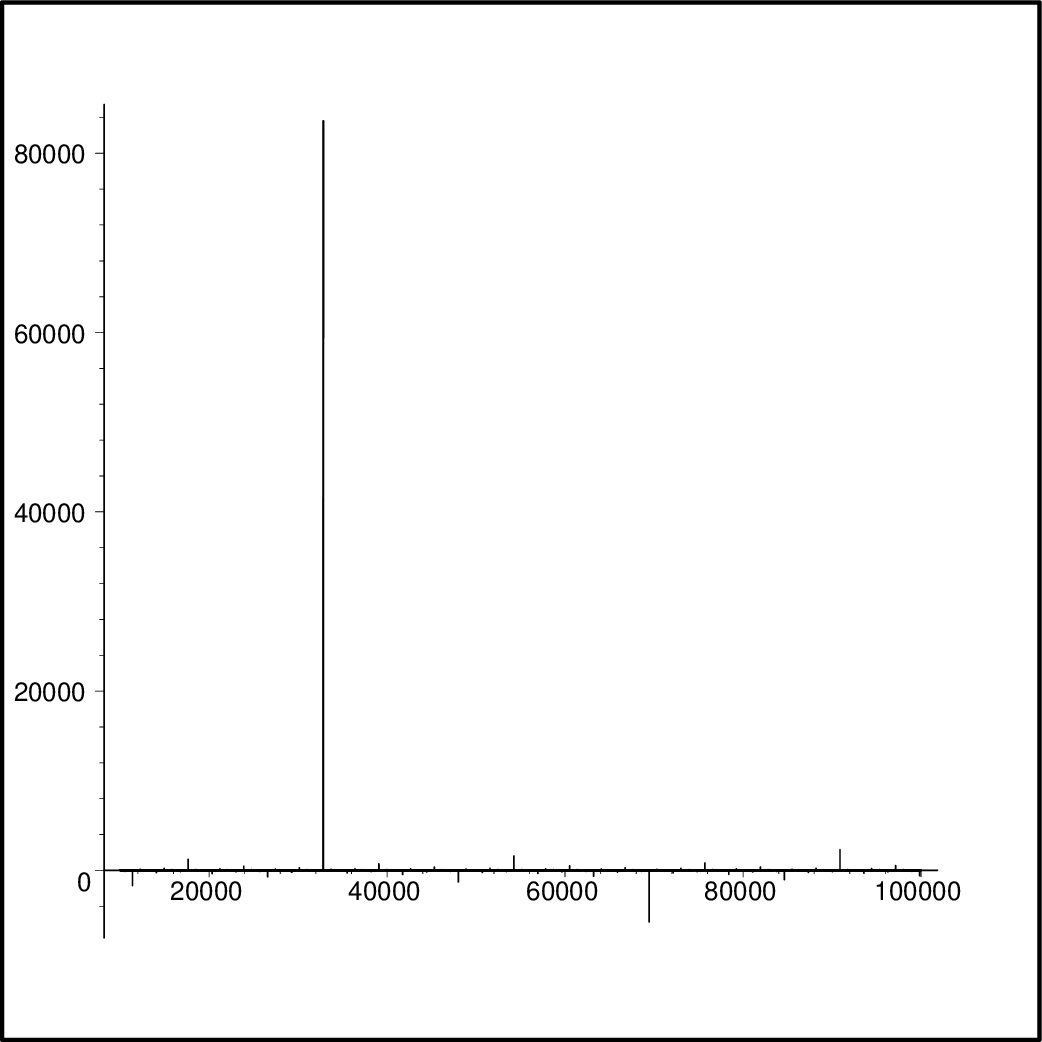}\end{center}
\caption{The values of $\mu_{D,n}([0,1/2))$ for (left) $n=0,\ldots,200$, (middle) $n=201,\ldots,10000$ and (right) $n=10001,\ldots,100000$.}
\label{fig:Dnoper}
\end{figure}

\begin{corollary} The sequence of the pure point measures $\mu_{D,n}$ does not converge weakly to a finite Borel measure.
\end{corollary} 

We now consider the dependence of Theorem \ref{BVexplicit} on the non-vanishing condition in Eq.~\eqref{zjcond}. While the dependence on this condition adds to the technicality of Theorem~\ref{BVexplicit} and to the difficulty in applying the result in full generality, seen from the view of some of common situations, the condition is quite natural. When ${\bf B}$'s dominant eigenvalue is simple, the non-vanishing condition in Eq.~\eqref{zjcond} is required to avoid certain degenerate situations. For example, consider the positive $k$-regular sequence $f$, where $f(m)=e_1^T\, {\bf B}_{(m)_k}\, {\bf w}$, with
$${\bf B}_0={\bf B}_1=\cdots={\bf B}_{k-1}=\left(\begin{matrix} * & 0\\ 0 & c \end{matrix}\right), \quad\mbox{and}\quad{\bf w}=\left(\begin{matrix} *\\ 0 \end{matrix}\right),$$ where the `$*$' indicates any positive values. For large integers $c$ the dominant eigenvalue of ${\bf B}$ will be $k\cdot c$, but none of the matrices ${\bf B}_a$ nor ${\bf w}$ have any component that interacts with this part of the matrix ${\bf B}$, and so this growth is not reflected in the behaviour of the sequence. Of course, the linear representation of the sequence is not unique, so this problem is not intrinsic to the sequence. More probable is that one has made a sub-optimal choice of the matrices ${\bf B}_a$---as is readily apparent, one can fix this example by deleting the un`$*$' parts.

With these examples in mind, assuming that degenerate situations can be avoided, the sufficient assumptions of Theorem \ref{BVexplicit} seem close to necessary for the existence of such measures. Of course, on a case-by-case basis, the measure could exist without satisfying these conditions, but it is certainly not guaranteed in general.

\section{Concluding remarks}\label{sec:cr}

In this paper, we have constructed measures associated with $k$-regular sequences and the vector spaces generated by their $k$-kernels in an effort to view these structures as a new source of dynamical objects, whose recursive properties are reminiscent of fractals. 

In working with examples, curiosities  having to do with spectral type are evident, and natural examples with certain properties are elusive. 

A case in point is the seeming scarcity of regular sequences yielding pure point measures. Of course, Theorem~\ref{main2} concludes that if $f$ is primitive (and other conditions), then $\mu_f$ is continuous. So to find examples, it seems wise to study sequences $f$ which have many zeros. For a trivial example, if one takes $f=\chi_{_2}$ to be the characteristic sequence of the powers of two, then we easily compute that $\mu_{\chi_{_2}}=\delta_0$. Staying with binary sequences for the moment, as soon as the density of ones in the fundamental regions $[2^n,2^{n+1})$ is bounded below, we are spreading the mass equally to these points, so the mass allocated to any one point in $\mathbb{T}$ must go to zero, and so assuming some regularity on the distribution of the non-zero values of the sequence the measure will be continuous; see Coons and Evans \cite{CEpre} for a family of singular continuous examples related to iterated function systems. Some non-trivial examples of sequences with pure-point measures exist, see Evans \cite{Epre}. However, these still occur in a relatively trivial way, where certain terms of the sequence contain a non-zero fraction of its entire mass. It seems likely that one could produce pure point measures differently, by instead having many small point masses concentrate in a single location, but it is unclear if this can be done in a non-trivial way within the confines of regular sequences; see Evans \cite[Thm.~1.1(2D)]{Epre}. For a non-regular example, consider the values of the binomial coefficients, say $2^n\choose m$ between powers of two and follow our process. The measure obtained will be $\delta_{1/2}$. This happens because, even though the ratio of every value to the sum is going to zero, the sequence is dominated by the central binomial coefficient and the limiting Gaussian curve, when scaled to the interval $[0,1)$, is collapsing to its mean.

Our paradigm contrasts starkly with the situation for diffraction measures, where there is a plethora of examples with pure point spectral type, e.g. those arising from model sets (both in the regular and weak sense) and Toeplitz systems; see Keller \cite{Keller} and Keller and Richard \cite{Keller3}. 
For example, the diffraction measure of the paperfolding sequence is pure point \cite[p.~380]{BGbook}, but in our construction the produced measure is Lebesgue. 
This is similar to the contrast between the spectral theory of Schr\"odinger operators and diffraction measures where absence of eigenvalues is also prevalent in the substitution setting; see Damanik, Embree and Gorodetski \cite{DEG2015}. 

The question of spectral purity also arises. In most of the examples we have produced, the measure $\mu_f$ is of pure type. As suggested by Theorem~\ref{main2}, focussing solely on continuous measures, is there a natural example of a regular sequence with continuous measure of mixed spectral type, that is, having both absolutely continuous and singular continuous components in its Lebesgue decomposition? While examples can be constructed---using that the set of $k$-regular sequences is a group under point-wise addition \cite{AS1992}---we have yet to find a natural non-trivial example or a ``good'' sufficient criterion for purity.

\section*{Acknowledgements}

It is our pleasure to thank Michael Baake for helpful conversations and comments, and for suggesting important references. M.~Coons and J.~Evans acknowledge the support of the Commonwealth of Australia and N.~Ma\~nibo acknowledges the support of the German Research Foundation (DFG) via the CRC 1283. 

\bibliographystyle{amsplain}
\providecommand{\bysame}{\leavevmode\hbox to3em{\hrulefill}\thinspace}
\providecommand{\MR}{\relax\ifhmode\unskip\space\fi MR }
\providecommand{\MRhref}[2]{%
  \href{http://www.ams.org/mathscinet-getitem?mr=#1}{#2}
}
\providecommand{\href}[2]{#2}


\end{document}